\documentclass[11pt]{article}

\usepackage{amsmath}
\usepackage{amssymb}
\usepackage{amsthm}
\usepackage{mathabx}
\usepackage{pgfplots}
\usepackage{tikz}
\usetikzlibrary{trees}
\usepackage{hyperref}
\usepackage[T2A]{fontenc}
\usepackage[utf8]{inputenc}
\usepackage{authblk}
\usepackage{datetime}
\DeclareMathOperator\supp{supp}
\DeclareMathOperator{\sign}{sgn}

\newfont{\bbf}{msbm10 scaled\magstep1}

\def\arxivprefix{https://arxiv.org}

\newcounter{glob}[section]
\renewcommand\theglob{%
	\ifnum\arabic{section}=0\else\arabic{section}.\fi %
	\arabic{glob}}

\newtheorem{thm}[glob]{Theorem}
\newtheorem{lem}[glob]{Lemma}
\newtheorem{cor}[glob]{Corollary}

\newtheorem{defi}[glob]{Definition}
\newtheorem{ex}[glob]{Example}
\newtheorem{propo}[glob]{Proposition}

\def\Z{\mathbb{Z}}

\def\N{\mathbb{N}}

\def\Fr{\Z[\frac{1}{2}]}

\newcommand*{\email}[1]{
	\normalsize\href{mailto:#1}{#1}\par}
\newcommand{\keywords}[1]{\textbf{\textit{Keywords---}} #1}

\pgfplotsset{width=8cm}

\newdateformat{UKvardate}{%
	\THEDAY\ \monthname[\THEMONTH], \THEYEAR}
\UKvardate

\title{\bf Convergence towards the end space for random walks on Schreier graphs}
\author{Bogdan Stankov}
\affil{D\'epartement de math\'ematiques et applications, \'Ecole normale sup\'erieure, CNRS,\\ PSL Research University, 75005 Paris, France\\ \email{bogdan.zl.stankov@gmail.com}}
\date{\today}

\begin{document}
\maketitle

\begin{abstract}We consider a transitive action of a finitely generated group $G$ and the Schreier graph $\Gamma$ defined by this action for some fixed generating set. For a probability measure $\mu$ on $G$ with a finite first moment we show that if the induced random walk is transient, it converges towards the space of ends of $\Gamma$. As a corollary we obtain that for a probability measure with a finite first moment on Thompson's group $F$, the support of which generates $F$ as a semigroup, the induced random walk on the dyadic numbers has a non-trivial Poisson boundary. Some assumption on the moment of the measure is necessary as follows from an example by Juschenko and Zheng.
\end{abstract}

\section{Introduction}
\thispagestyle{empty}

Consider a finitely generated group $G$ acting on a space $X$ (on the right). For a point $x\in X$ and a generating set $S$, the Schreier graph $\Gamma=(xG,E)$ is the graph the vertex set of which is the orbit $xG$ of $x$, and the edges $E$ are the couples of the form $(y,y.s)$ for $y\in xG$ and $s\in S$. Throughout this article, we will assume the action to be transitive, that is for every $x$, $xG=X$. We take a measure $\mu$ on $G$ and will study for which $(G,\Gamma,\mu)$ the induced random walk on $\Gamma$ converges towards an end of the graph. We recall the definition of the end space. Consider an exhaustive increasing sequence $K_1\subset K_2\subset\dots$ of finite subsets of $X$. An \textit{end} of $\Gamma$ is a sequence $U_1\supseteq U_2\supseteq\dots$ where $U_n$ is an infinite connected component of the subgraph obtained by deleting the vertices in $K_n$ and adjacent edges. For more details, see Definition~\ref{end}. Our main result states:

\begin{thm}\label{main}
	Consider a finitely generated group $G$ acting transitively on a space $X$. Fix a generating set $S$ and let $\Gamma=(X,E)$ be the associated Schreier graph. Let $\mu$ be a measure on $G$ with a finite first moment such that the induced random walk on $\Gamma$ is transient. Then the random walk almost surely converges towards a (random) end of the graph.
\end{thm}

Notice that for measures with finite support, the result is straightforward. The result is also already known in the case where the action of $G$ on $X$ is non-amenable (this is a particular case of \cite[Theorem~21.16]{woess2000random}, which we recall as Theorem~\ref{woess}), under the condition of a finite first moment. An action is non-amenable when there is no $G$-invariant mean on $X$. Kesten's criterion~\cite{Kesten1959} states that for any symmetric non-degenerate measure on the group, the action is non-amenable if and only if the induced random walk on $X$ has probability of return to the origin that decreases exponentially (see Bartholdi~\cite{bartholdi} for a survey on the amenability of group actions). The general case of the cited~\cite[Theorem~21.16]{woess2000random} does not assume that the random walk is induced by a measure on a group. The result is no longer true if we assume neither that the walk is induced by a measure on a group nor that the probability of return to the origin decreases exponentially. We prove that in Proposition~\ref{counterexample}, where we construct a Markov chain $(X,P)$ that is transient, uniformly irreducible and has uniform first moment, but does not converge towards an end of $X$. 

If the action is non-amenable, the random walk induced by any non-degenerate measure is transient (see~\cite[Lemma~1.9]{woess2000random}). In the general case, transience can sometimes be obtained from the graph geometry using a comparison Lemma~\ref{var} due to Baldi-Lohoué-Peyrière~\cite{var}. Combining this lemma and the theorem we obtain:

\begin{cor}\label{all}
	Consider a finitely generated group $G$ acting transitively on a space $X$. Fix a generating set $S$ and let $\Gamma=(X,E)$ be the associated Schreier graph. Assume that $\Gamma$ is a transient graph. Then for all measures $\mu$ on $G$ with finite first moments such that $\supp(\mu)$ generates $G$ as a semigroup, the induced random walk almost surely converges towards an end of the graph.
\end{cor}

We will also explain how this result can be applied to Thompson's group $F$. Let us recall the definition of this group. The set of dyadic rationals $\Fr$ is the set of numbers of the form $a2^b$ with $a,b\in\Z$. \textit{Thompson's group $F$} is the group of orientation-preserving piecewise linear self-isomorphisms of the closed unit interval with dyadic slopes, with a finite number of break points, all break points being in $\Fr$. It is a finitely generated group with a canonical generating set (with two elements). See Cannon-Floyd-Parry~\cite{thomsoncfp} or Meier's book~\cite[Ch.~10]{meier} for details and properties. Its amenability is a celebrated open question. It is well known that amenability is equivalent to the existence of a non-degenerate measure with trivial Poisson boundary (see Kaimanovich-Vershik~\cite{kaimpoisson}, Rosenblatt~\cite{rosenblatt}). The boundary of a random walk induced by an action is a quotient of the boundary on the group.

The Schreier graph on $\Fr$ (of a conjugate action of $F$) has been described by Savchuk~\cite[Proposition~1]{slav10}. It is a tree that can be understood as a combination of a skeleton quasi-isometric to a binary tree, and rays attached at each point of the skeleton (see Figure~\ref{sav}). Understanding the geometry of the graph directly shows that it is transient. Kaimanovich~\cite[Theorem~14]{kaimanovichthompson} also proves this result without using the geometry of the graph. Hence by Corollary~\ref{all} we obtain

\begin{cor}
	Consider a measure on Thompson's group $F$ with a finite first moment, the support of which generates $F$ as a semigroup. Then the induced random walk on $\Fr$ has non-trivial Poisson boundary.
\end{cor}

This extends the following previous results. Kaimanovich~\cite{kaimanovichthompson} and Mishchenko~\cite{mischenko2015} prove that the simple random walk on the Schreier graph given by that action has non-trivial boundary. Kaimanovich~\cite[Section~6.A]{kaimanovichthompson} further shows that it is non-trivial for walks induced by measures with supports that are finite and generate $F$ as a semigroup. We have also shown~\cite{hz} that for any measure with a finite first moment on $F$, the support of which generates $F$ as a semigroup, the walk on the group has non-trivial Poisson boundary.

The result of the corollary is false without assuming a finite first moment. Juschenko and Zheng~\cite{juszheng} have proven that there exists a symmetric non-degenerate measure on $F$ such that the induced random walk has trivial Poisson boundary. If the trajectories almost surely converge towards points on the end space, the end space endowed the exit measure on it is a quotient of the Poisson boundary. However, the self-similarity of the graph implies that the exit measure cannot be trivial, as we prove in Lemma~\ref{nontriv}. Combining the result of Juschenko-Zheng with this lemma we obtain:

\begin{cor}\label{juscor}
	There exists a finitely generated group $G$, a space $X$ and a symmetric non-degenerate measure on $G$ such that
	\begin{itemize}
		\item $G$ acts amenably and transitively on $X$,
		\item the induced random walk on the Schreier graph is transient,
		\item the induced random walk on the Schreier graph does not converge towards an end of the graph.
	\end{itemize}
	In particular, the measure described by Juschenko and Zheng~\cite[Theorem~3]{juszheng} provides an example for the action of Thompson's group $F$ on $\Fr$.
\end{cor}

Concerning Thompson's group $F$, studying the Poisson boundary of random walks on it has been highlighted as a possible approach to proving non-amenability in the work of Kaimanovich. The results by him and Mischenko further suggested that one could consider the boundary of induced random walks $\Fr$, but that was shown impossible by the result of Juschenko-Zheng. In more recent results, Juschenko~\cite{juscof} studied walks on the space of $n$-element subsets of $\Fr$ and gave a combinatorial necessary and sufficient condition for the Poisson boundary of induced walks on that space to be non-trivial for all non-degenerate measures. In that situation, the existence a measure with trivial boundary is due to Juschenko for $n=2$ and to Schneider and Thom~\cite{Schneider2020} for a general $n$.

\section*{Acknowledgements}

I would like to thank Professor Wolfgang Woess for a discussion on the early variants of this paper which helped to improve the original version. I am very grateful to Professor Vadim Kaimanovich for his detailed comments and useful suggestions.

\section{Preliminaries}

Consider a finitely generated group $G$ acting transitively on a space $X$ and a measure $\mu$ on $G$. The random walk on $G$ is defined by multiplication on the right. That is the walk with trajectories $(g_n)$ for $n\in\N$ where $g_{n+1}=g_nh_n$ and the increments $h_n$ are sampled by $\mu$. In other words, the random walk is defined by the kernel $(g,h)\mapsto\mu(g^{-1}h)$. The trajectory of the induced random walk on $X$ starting at a point $\mathfrak{o}$ is the image of the trajectory of the walk on the group by the map:

$$(g_n)\mapsto (\mathfrak{o}.g_n).$$

Its kernel is $P(x,y)=\sum_{x.g=y}\mu(g)$. We now fix a generating set $S$ of $G$ and consider the undirected graph $\Gamma=(X,E)$ with vertices $X$ and edges $E=\{(x,x.s)\mbox{ for }s\in S,x\in X\}$. We recall that this is called the \textit{Schreier graph}, and that it is connected as we assumed the action to be transitive. It is worth noting that the directed version of the same definition is also referred to as the Schreier graph, and that in the figures in this article, the edges will have an assigned direction for easier visualisation. It is known that every connected regular graph of even degree is isomorphic to a Schreier graph. It was first proven by Gross~\cite{MR0450121} for finite graphs. For a detailed proof of the infinite case, see~\cite[Theorem~3.2.5]{thesis}. For a more in-depth study of Schreier graphs, see Grigorchuk-Kaimanovich-Nagnibeda~\cite{Grigorchuk2012}.

\begin{defi}\label{end}
	For a compact $K\subset X$ denote by $\pi_0(X\setminus K)$ the set of connected components of $X\setminus K$. There is a natural partial order defined by $K_1\leq K_2$ if and only if $K_1\subseteq K_2$. That order gives rise to a natural morphism $\pi_{1,2}:\pi_0(X\setminus K_2)\mapsto \pi_0(X\setminus K_1)$ which sends a connected component into the connected component of which it is a subset. This forms an inverse system indexed by $K\subset X$ (see~\cite[Section~3.1.2]{qtheory}). The end space is then the inverse limit
	
	$$\lim_{K\subset\Gamma}\pi_0(X\setminus K)=\{(x_K)\in\prod_{K\subset X}\pi_0(X\setminus K)|\pi_{1,2}x_2=x_1,\ K_1\subset K_2\}.$$
\end{defi}

In our case, the end space can be described using an increasing exhaustive sequence of finite sets $K_n$, as such sequences are cofinal in the set of all compact subsets. That is, any compact set is included in $K_n$ for $n$ large enough.

We use the following comparison lemma by Baldi-Lohoué-Peyrière~\cite{var}.

\begin{lem}[Comparison lemma]\label{var}
	Let $P_1(x,y)$ and $P_2(x,y)$ be doubly stochastic kernels on a countable set $X$ and assume that $P_2$ is symmetric. Assume that there exists $\varepsilon\geq 0$ such that
	
	$$P_1(x,y)\geq\varepsilon P_2(x,y)$$
	for any $x,y$. Then if $P_2$ is transient then so is $P_1$.
\end{lem}

Here doubly stochastic kernels means that the operators are reversible and the inverse of each is also Markov. Equivalently, they preserve the counting measure; it is worth noting that the result holds true more generally for operators with any common stationary measure, see Kaimanovich~\cite[Section~3.C]{kaimanovichthompson}; see also Woess~\cite[Section~2.C,3.A]{woess2000random}. For the walks considered in this article, it is direct to verify that, for all probability measures, $p(x,y)=\mu(x^{-1}y)$ is doubly stochastic (as the inverse operator is defined by $(x,y)\mapsto\mu(y^{-1}x)$).

We recall that a random walk is called \textit{transient} if, for any point, almost every trajectory leaves that point after finite time. Otherwise, the walk is called \textit{recurrent} and there is a point that the walk almost surely visits an infinite amount of times. A graph is called transient (recurrent) if the simple random walk on it is transient (recurrent). The \textit{Green function} $G$ is defined by $G_z(x,y)=\sum_{n\in\N}p^{(n)}(x,y)z^n$ where $p^{(n)}$ is the $n$-time transition probability of $p$. In other words, $p^{(n)}(x,y)$ is the probability that a random walk starting in $x$ is at $y$ after $n$ steps. We will write $G(x,y)=G_1(x,y)$. A walk is transient if and only if $G(x,x)<\infty$ for all $x\in X$.

Remark that recurrent walks do not converge to the end space. However, it is possible for a measure on a group to induce a transient walk even if the uniform measure is recurrent, in which case we can apply Theorem~\ref{main}. Here we give an example of that situation in which the graph has infinitely many ends.

\begin{ex}
	Consider the graph $\Psi$ in Figure~\ref{recc}. Consider the action of the free group on two generators $F_2$ on it where the first generator $a$ sends each vertex to the right, and the second generator $b$ sends a vertex to the vertex above if possible, and to itself otherwise. The graph is recurrent. Consider the measure $\mu(a)=\frac{3}{8}$, $\mu(a^{-1})=\frac{1}{8}$, $\mu(b)=\mu(b^{-1})=\frac{1}{4}$. It is transient and converges towards the ends defined by the right rays.
\end{ex}

\begin{figure}[!h]\centering\caption{A recurrent graph with infinitely many ends}\label{recc}\begin{tikzpicture}
		\tikzset{no edge/.style={edge from parent/.append style={draw=none}}}
		\tikzset{node/.style={circle,draw,inner sep=0.7,fill=black}}
		\tikzset{every loop/.style={min distance=8mm,in=55,out=125,looseness=10}}
		%\tikzstyle{level 0}=[level distance=10mm,sibling distance=25mm]
		\tikzstyle{level 1}=[level distance=2.4cm,sibling distance=3cm]
		%	\tikzstyle{level 2}=[level distance=2.4cm,sibling distance=3mm]
		%	\tikzstyle{level 3}=[level distance=1.5cm,sibling distance=5mm]
		%	\tikzstyle{level 4}=[level distance=1cm,sibling distance=5mm]
		
		\node[node,label=above:{\dots}](){}
		child[grow=down,level distance=1cm]{node[node]{}
			child[grow=left,level distance=2.4cm]{node[node](12){}child[grow=left,level distance=2.4cm]{node[node,label=left:{\dots}](14){}}}
			child[grow=right,level distance=2.4cm]{node[node](2){}child[grow=right,level distance=2.4cm]{node[node,label=right:{\dots}](4){}}}
			child[grow=down,level distance=1cm]{node[node]{}
				child[grow=left,level distance=2.4cm]{node[node](5){}child[grow=left,level distance=2.4cm]{node[node,label=left:{\dots}](6){}}}
				child[grow=right,level distance=2.4cm]{node[node](7){}child[grow=right,level distance=2.4cm]{node[node,label=right:{\dots}](9){}}}
				child[grow=down,level distance=1cm]{node[node]{}
					child[grow=left,level distance=2.4cm]{node[node](10){}child[grow=left,level distance=2.4cm]{node[node,label=left:{\dots}](11){}}}
					child[grow=right,level distance=2.4cm]{node[node](4b){}child[grow=right,level distance=2.4cm]{node[node,label=right:{\dots}](6b){}}}
					child[grow=down,level distance=1cm]{node[node,label=below:{\dots}]{}}
				}
			}
		};
		\draw (12) edge[loop above] (12);
		\draw (14) edge[loop above] (14);
		\draw (2) edge[loop above] (2);
		\draw (4) edge[loop above] (4);
		\draw (4b) edge[loop above] (4b);
		\draw (5) edge[loop above] (5);
		\draw (6) edge[loop above] (6);
		\draw (6b) edge[loop above] (6b);
		\draw (7) edge[loop above] (7);
		\draw (9) edge[loop above] (9);
		\draw (10) edge[loop above] (10);
		\draw (11) edge[loop above] (11);
\end{tikzpicture}\end{figure}
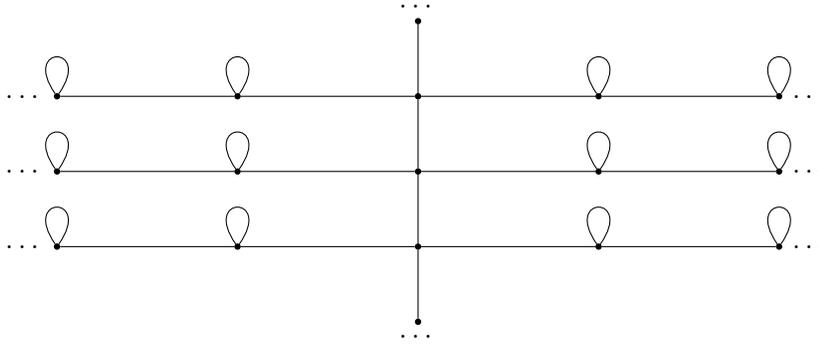

If we do not require the measure on $F_2$ to have a finite first moment, it can be chosen symmetric while the induced walk remains transient. This can be done on any graph containing an infinite array, see~\cite[Lemma~7.1]{erschsubexp}. Furthermore, we can construct recurrent graphs for which it is possible to have symmetric measures (on the acting group) with finite first moments that induce transient walks:

\begin{ex}
	Consider the graph $\Psi'$ obtained by $\Psi$ by replacing the horizontal lines with $\Z^2$ planes. It is a recurrent graph. Consider the free product $\Z\ast\Z^2$ with generators $a\in\Z$ and $b,c\in\Z^2$. Consider its action on $\Psi'$ where $a$ moves a vertex to the vertex above if possible, and to itself otherwise, and $b$ and $c$ act horizontally. There is a symmetric transient measure $\mu$ on $\Z^2$ with a finite first moment. Consider $\mu'=\frac{1}{4}(\delta_a+\delta_{a^{-1}})+\frac{1}{2}\mu$. It induces a transient walk on $\Psi'$, which by Theorem~\ref{main} almost surely converges to an element of end space.
\end{ex}

Let us recall the exact statement of Theorem~21.16 from the book of Woess~\cite{woess2000random}. For a graph $X$ and a Markov operator $P$ on it, the theorem states:

\begin{thm}[{{\cite[Theorem~21.16]{woess2000random}}}]\label{woess}
	If $(X,P)$ is uniformly irreducible and has a uniform first moment, and $\rho(P)<1$, then the random walk defined by $(X,P)$ converges almost surely to a random end of $X$.
\end{thm}

Let us define the concepts in the statement. The walk is uniformly irreducible if there exists $c>0$ and finite $K\in\N$ such that for all neighbouring vertices $x$ and $y$, there exists $k\leq K$ such that $p^{(k)}(x,y)\geq c$. The step distribution on a point $x\in X$ is defined as $\sigma_x(n)=\sum_{y:d(x,y)=n}p(x,y)$. The step distributions are \textit{tight} if there is a distribution $\sigma$ on $\N_0$, such that for all $x$ and all $n$, the tails $\sigma_x([n,+\infty))$ are bounded by the tails of $\sigma$. The walk has uniform first moment if the step distributions are tight with some $\sigma$ that has finite first moment. The spectral radius is $\rho(P)=\limsup_{n\rightarrow\infty}p^{(n)}(x,y)^{1/n}$ (this quantity does not depend on $x$ and $y$). It is straightforward to check that if $\rho(P)<1$, then the random walk is transient. Moreover, applying the definition for $x=y$ we see that $\rho(P)<1$ if any only if the probability of return to the origin decreases exponentially. We will show that the result of Theorem~\ref{woess} is not true without the assumption $\rho(P)<1$. By $\sign$ we denote the sign function on $\Z$: $\sign(z)=1$ if $x\geq0$ and $\sign(x)=-1$ if $x<0$.

\begin{propo}\phantomsection\label{counterexample}
\begin{enumerate}
\item There is a graph $X$ and a Markov operator $P$ on it such that $(X,P)$ is transient, uniformly irreducible and has a uniform first moment, but the random walk defined by $(X,P)$ doesn't converge almost surely to a random end of $X$.

\item Consider the Markov chain $(\Z,P_{p_n,\varepsilon_n})$ which, given $p_n\geq0$ and $\varepsilon_n\geq0$, is defined as

\begin{equation*}
	P(x,y)=\begin{cases}
		(1-p_n)\frac{1+\varepsilon_n}{2} &\mbox{for } y=\sign(x)(|x|+1)\\
		(1-p_n)\frac{1-\varepsilon_n}{2} &\mbox{for } y=\sign(x)(|x|-1)\\
		p_n &\mbox{for }y=-x\\
		0&\mbox{otherwise.}
	\end{cases}
\end{equation*}

There is a choice of $p_n\geq0$ and $\varepsilon_n\geq0$ such that $(\Z,P_{p_n,\varepsilon_n})$ is transient, uniformly irreducible, has uniform first moment and has an infinite expected number of steps where the sign changes. In particular, it verifies the conditions of (1).
\end{enumerate}
\end{propo}

The exact values that appear in the proof are $p_n=\frac{1}{n^2(\ln n)^2}$ and $\varepsilon_n=\frac{(n+1)(\ln(n+1))^2-n(\ln n)^2}{(n+1)(\ln(n+1))^2+n(\ln n)^2}$.

\begin{proof}
	We will find sufficient sufficient conditions on $p_n\geq0$ and $\varepsilon_n\geq0$ under which $(\Z,P_{p_n,\varepsilon_n})$ verifies the conditions we seek, and then provide a choice that satisfies those conditions. Specifically, the sufficient conditions are (\ref{pnfirstmoment}),(\ref{unifirr}),(\ref{transient}) and (\ref{jumps}).
	
	The tails $\sigma_x([n,+\infty))$ are bounded by the tail of the distribution $\sigma$ on $\N_0$ defined by $\sigma(0)=\sigma(1)=1$, $\sigma(2n)=p_n$ for $n\geq1$ and $\sigma(x)=0$ otherwise. The Markov chain $(\Z,P_{p_n,\varepsilon_n})$ has uniform first moment if and only if $\sigma$ has finite first moment, or equivalently
	
	\begin{equation}\label{pnfirstmoment}
		\sum_{n\in\N} np_n<\infty.
	\end{equation}
	
	For $(\Z,P_{p_n,\varepsilon_n})$ to be uniformly irreducible, it would suffice that there should exist $c>0$ such that $(1-p_n)\frac{1-\varepsilon_n}{2}\geq c$ for all $n$. If we have
	\begin{equation}\label{unifirr}
		p_n\xrightarrow{n\rightarrow\infty}0\mbox{ and }\varepsilon_n\xrightarrow{n\rightarrow\infty}0
	\end{equation}
	then $(1-p_n)\frac{1-\varepsilon_n}{2}\xrightarrow{n\rightarrow\infty}\frac{1}{2}$. In that case, replacing if necessary the values a finite number of $p_n$ and/or $\varepsilon_n$ with $0$, we can have $(1-p_n)\frac{1-\varepsilon_n}{2}\geq c$.
	
	To study the transience of $(\Z,P_{p_n,\varepsilon_n})$ we consider $\widetilde{P}$ on $\N_0$ defined by $\widetilde{P}(k,k+1)=(1-p_n)\frac{1+\varepsilon_n}{2}$, $\widetilde{P}(k,k-1)=(1-p_n)\frac{1-\varepsilon_n}{2}$ and $\widetilde{P}(k,k)=p_n$. It is a nearest neighbour random walk on $\N_0$ and its transience is equivalent to the transience of $(\Z,P_{p_n,\varepsilon_n})$. Nearest neighbour random walks on $\N_0$ are well understood. As seen in \cite[Section~2.16]{woess2000random}, $(\N_0,\widetilde{P})$ is transient if and only if
	\begin{equation}\label{transient}
		\sum_{k=1}^\infty r(e_k)<\infty
	\end{equation}
	where $r(e_k)=\frac{\widetilde{P}(k-1,k-2)\dots\widetilde{P}(1,0)}{\widetilde{P}(0,1)\dots\widetilde{P}(k-1,k)}$. We have $\frac{r(e_{k+1})}{r(e_k)}=\frac{1-\varepsilon_k}{1+\varepsilon_k}$ and therefore defining $\varepsilon_k$ is equivalent to defining $r(e_k)$.
	
	Finally, if the Green function of $\widetilde{P}$ is $G^{(\widetilde{P})}$, then the expected number of "jumps" between $n$ and $-n$ is $G^{(\widetilde{P})}(n,n)p_n$. We wish to obtain $\sum_nG^{(\widetilde{P})}(n,n)p_n=\infty$. From the results of \cite[Example~2.13,~Section~2.16]{woess2000random} it follows that $G^{(\widetilde{P})}(n,n)=\frac{1}{r(e_n)\widetilde{P}(n,n-1)}\sum_{k=n+1}^\infty r(e_k)$. If $\widetilde{P}(n,n-1)\geq c$, it would suffice to have
	
	\begin{equation}\label{jumps}
		\sum_{n\in\N}p_n\frac{1}{r(e_n)}\sum_{k=n+1}^\infty r(e_k)=\infty.
	\end{equation}
	
	We now define $r(e_k)$ and $p_k$ and claim that those choices verify conditions (\ref{pnfirstmoment}),(\ref{unifirr}),(\ref{transient}) and (\ref{jumps}). Let
	
	$$r(e_k)=\frac{1}{k(\ln k)^2}\mbox{ and }p_k=\frac{1}{k^2(\ln k)^2}.$$
	
	We first prove condition (\ref{pnfirstmoment}). It suffices to observe that
	
	$$\sum_{n\geq2} np_n\leq\int_1^\infty\frac{x}{x^2(\ln x)^2}dx=-\frac{1}{\ln x}\Biggr|_1^\infty,$$
	which is finite. As $r(e_k)=kp_k$, this also proves condition (\ref{transient}). Condition (\ref{unifirr}) is straightforward.
	
	We now only need to prove condition (\ref{jumps}). Similarly, we have
	$$\sum_{k=n+1}^\infty r(e_k)\geq\int_{n+1}^\infty\frac{1}{x(\ln x)^2}dx=-\frac{1}{\ln x}\Biggr|_{n+1}^\infty$$
	and thus
	$$\frac{1}{r(e_n)}\sum_{k=n+1}^\infty r(e_k)\geq\frac{n(\ln n)^2}{\ln(n+1)}\approx n\ln n.$$
	
	Then
	$$\sum_{n\in\N}p_n\frac{1}{r(e_n)}\sum_{k=n+1}^\infty r(e_k)\geq
	%\sum_{n\in\N}\frac{1}{n^2(\ln n)^2}\frac{n(\ln n)^2}{\ln(n+1)}=
	\sum_{n\in\N}\frac{1}{n\ln(n+1)}\geq\int_2^\infty\frac{1}{x\ln x}dx=\ln(\ln(x))\Biggr|_2^\infty$$
	which is not finite.
\end{proof}

\section{Proof of main Theorem~\ref{main}}

Consider a finite set $K\subset X$ and denote $\Gamma_1,\dots,\Gamma_k$ the connected components of its complement. We will study the probability to change the component at step $n$ and prove that the sum over $n$ is finite. 

Consider $x\in X\setminus K$ and $g\in G$. We will study the probability that $x.g$ is not in the same component. Let $g=s_1s_2\dots s_n$ where $|g|=n$ and $s_i\in S$. If $x$ and $x.g$ are in different components, by definition the path $x,x.s_1,\dots,x.g$ passes trough $K$. Therefore there is $i$ such that $x.s_1s_2\dots s_i\in K$. Equivalently, $\langle\sum_{i\leq n}\tau_{s_1s_2\dots s_i}\delta_x,\sum_{k\in K}\delta_k\rangle\geq1$ where $\delta_y$ is the characteristic function at a given point $y$ and $\tau_f$ is the translation defined by $\tau_f\delta_y=\delta_{y.f}$. We observe

$$\langle\sum_{i\leq n}\tau_{s_1s_2\dots s_i}\delta_x,\sum_{k\in K}\delta_k\rangle=\langle\delta_x,\sum_{i\leq n}\sum_{k\in K}\tau_{s_i^{-1}\dots s_2^{-1}s_1^{-1}}\delta_k\rangle.$$

We denote

$$f=\sum_{s_1s_2\dots s_n\in G}\mu(s_1s_2\dots s_n)\sum_{i\leq n}\sum_{k\in K}\tau_{s_i^{-1}\dots s_2^{-1}s_1^{-1}}\delta_k.$$
Then the probability that $x$ and $x.g$ are in different components is not greater than $\langle\delta_x,f\rangle$. Furthermore, the $l^1$ norm of $f$ satisfies $\|f\|_1\leq|K|\|\mu\|_1$ where $\|\mu\|_1$ is the first moment of $\mu$. In particular, it is finite.

Take a random walk starting at a fixed point $\mathfrak{o}$ and consider $n$ large enough so that the transient walk has left $K$. The probability of changing component at step $n+1$ is then not greater than

$$\langle p^{(n)}\delta_\mathfrak{o},f\rangle.$$

We have

$$\sum_{n\in\N}\langle p^{(n)}\delta_\mathfrak{o},f\rangle=\sum_{n\in\N}\sum_{x\in X}p^{(n)}(\mathfrak{o},x)f(x)=\sum_{x\in X}f(x)G(\mathfrak{o},x)$$
where we will have the right to interchange the order of summation if we prove that the right-hand side is finite. Let $\check{p}$ be the kernel induced by the inverse measure $\check{\mu}:g\mapsto\mu(g^{-1})$, and $G^{(\check{p})}$ the Green function corresponding to that kernel. Then $G(\mathfrak{o},x)=G^{(\check{p})}(x,\mathfrak{o})$. It is a known property of the Green function that for all $x,y\in X$, we have $G^{(\check{p})}(x,y)\leq G^{(\check{p})}(y,y)$. This follows from the fact that the left hand side is the expected number of visits of $y$ of a walk starting at $x$, while the right hand side is the expected number of visits starting at $y$. Thus

$$\sum_{x\in X}f(x)G(\mathfrak{o},x)\leq G^{(\check{p})}(\mathfrak{o},\mathfrak{o})\|f\|_1<\infty.$$

This proves that after finite time, the walk almost surely stays in the same connected component of the complement of $K$. Applying this for an increasing exhaustive sequence of $K$, we obtain the result of Theorem~\ref{main}.

It is worth mentioning that this approach is similar to the one used by Kaimanovich~\cite[Theorem~3.3]{Kaimanovich1991} to prove pointwise convergence of the configuration of walks on lamplighter groups with a finite first moment.

\section{Thompson's group $F$}\label{thompsect}

We now apply Theorem~\ref{main} to Thompson's group $F$. The Schreier graph on the dyadic numbers has been described by Savchuk~\cite[Proposition~1]{slav10}(see Figure~\ref{sav}). We have the following self-similarity result:

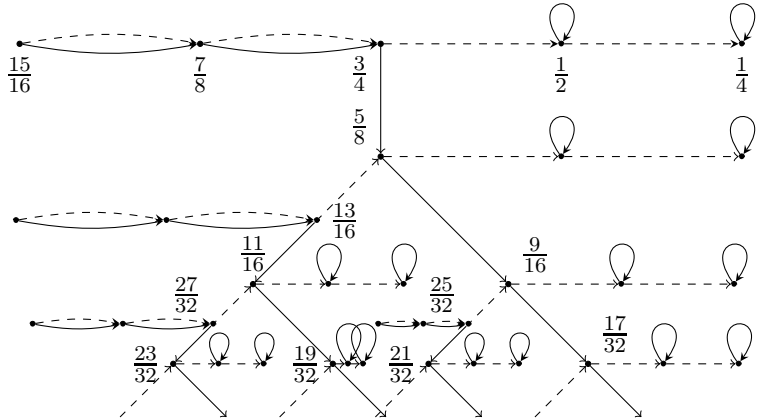
\begin{figure}[!h]\centering\caption{Schreier graph of the dyadic action of $F$ for the standard generators}\label{sav}\begin{tikzpicture}[-stealth]
		\tikzset{no edge/.style={edge from parent/.append style={draw=none}}}
		\tikzset{node/.style={circle,draw,inner sep=0.7,fill=black}}
		\tikzset{every loop/.style={min distance=8mm,in=55,out=125,looseness=10}}
		%\tikzstyle{level 0}=[level distance=10mm,sibling distance=25mm]
		\tikzstyle{level 1}=[level distance=2.4cm,sibling distance=3cm]
		\tikzstyle{level 2}=[level distance=2.4cm,sibling distance=12mm]
		\tikzstyle{level 3}=[level distance=1.5cm,sibling distance=5mm]
		\tikzstyle{level 4}=[level distance=1cm,sibling distance=5mm]
		
		\node[node,label=south west:{$\frac{3}{4}$}](34){}
		child[grow=left,<-,level distance=2.4cm]{[no edge] node[node,label=below:{$\frac{7}{8}$}](78){}child[grow=left,<-,level distance=2.4cm]{[no edge] node[node,label=below:{$\frac{15}{16}$}](1516){}}}
		child[grow=right,level distance=2.4cm,dashed]{node[node,label=below:{$\frac{1}{2}$}](12){}child[grow=right,level distance=2.4cm,dashed]{node[node,label=below:{$\frac{1}{4}$}](14){}}}
		child[grow=down,->,level distance=1.5cm]{node[node,label=north west:{$\frac{5}{8}$}]{}
			child[grow=south west,<-,dashed,level distance=1.2cm]{node[node,label=right:{$\frac{13}{16}$}](1316){} 
				child[grow=left,<-,level distance=2cm]{[no edge] node[node](1316a){}child[grow=left,<-,level distance=2cm]{[no edge] node[node](1316b){}}}
				child[grow=south west,->,solid,level distance=1.2cm]{node[node,label=above:{$\frac{11}{16}$}](1116){}
					child[grow=south west,<-,dashed,level distance=7.5mm]{node[node,label=north west:{$\frac{27}{32}$}](2732){}
						child[grow=left,<-,level distance=1.2cm]{[no edge] node[node](2732a){}child[grow=left,<-,level distance=1.2cm]{[no edge] node[node](2732b){}}}
						child[grow=south west,->,solid,level distance=7.5mm]{node[node,label=left:{$\frac{23}{32}$}](2332){}
							child[grow=south west,<-,dashed,level distance=1cm]{}
							child[grow=south east,->,solid,level distance=1cm]{}
							child[grow=right,->,dashed,level distance=6mm]{node[node](1){} child[grow=right,level distance=6mm]{node[node](8){}}}}}
					child[grow=south east,->,solid,level distance=1.5cm]{node[node,label=left:{$\frac{19}{32}$}]{}
						child[grow=south west,<-,dashed,level distance=1cm]{}
						child[grow=south east,->,solid,level distance=1cm]{}
						child[grow=right,->,dashed,level distance=0.2cm]{node[node](9){} child[grow=right,level distance=0.2cm]{node[node](10){}}}}
					child[grow=right,->,dashed,level distance=1cm]{node[node](2){} child[grow=right,level distance=1cm]{node[node](11){}}}}}
			child[grow=south east,->,solid,level distance=2.4cm]{node[node,label=north east:{$\frac{9}{16}$}]{}
				child[grow=south west,<-,dashed,level distance=7.5mm]{node[node,label=north west:{$\frac{25}{32}$}](2532){}
					child[grow=left,<-,level distance=0.6cm]{[no edge] node[node](2532a){}child[grow=left,<-,level distance=0.6cm]{[no edge] node[node](2532b){}}}
					child[grow=south west,->,solid,level distance=7.5mm]{node[node,label=left:{$\frac{21}{32}$}](2332){}
						child[grow=south west,<-,dashed,level distance=1cm]{}
						child[grow=south east,->,solid,level distance=1cm]{}
						child[grow=right,->,dashed,level distance=0.6cm]{node[node](3){} child[grow=right,level distance=0.6cm]{node[node](3b){}}}}}
				child[grow=south east,->,solid,level distance=1.5cm]{node[node,label=north east:{$\frac{17}{32}$}]{}
					child[grow=south west,<-,dashed,level distance=1cm]{}
					child[grow=south east,->,solid,level distance=1cm]{}
					child[grow=right,->,dashed,level distance=1cm]{node[node](4){} child[grow=right,level distance=1cm]{node[node](4b){}}}}
				child[grow=right,->,dashed,level distance=1.5cm]{node[node](6){} child[grow=right,level distance=1.5cm]{node[node](6b){}}}}
			child[grow=right,->,dashed,level distance=2.4cm]{node[node](5){} child[grow=right, level distance=2.4cm]{node[node](7){}}}
		};
		\draw (1516) edge[bend right=10] (78);
		\draw (78) edge[bend right=10] (34);
		\draw (1516) edge[bend left=10,dashed] (78);
		\draw (78) edge[bend left=10,dashed] (34);
		\draw (1316b) edge[bend right=10] (1316a);
		\draw (1316a) edge[bend right=10] (1316);
		\draw (1316b) edge[bend left=10,dashed] (1316a);
		\draw (1316a) edge[bend left=10,dashed] (1316);
		\draw (2732b) edge[bend right=10] (2732a);
		\draw (2732a) edge[bend right=10] (2732);
		\draw (2732b) edge[bend left=10,dashed] (2732a);
		\draw (2732a) edge[bend left=10,dashed] (2732);
		\draw (2532b) edge[bend right=10] (2532a);
		\draw (2532a) edge[bend right=10] (2532);
		\draw (2532b) edge[bend left=10,dashed] (2532a);
		\draw (2532a) edge[bend left=10,dashed] (2532);
		\draw (12) edge[loop above] (12);
		\draw (14) edge[loop above] (14);
		\draw (1) edge[loop above,min distance=6mm,in=55,out=125,looseness=10] (1);
		\draw (2) edge[loop above] (2);
		\draw (3) edge[loop above,min distance=6mm,in=55,out=125,looseness=10] (3);
		\draw (3b) edge[loop above,min distance=6mm,in=55,out=125,looseness=10] (3b);
		\draw (4) edge[loop above] (4);
		\draw (4b) edge[loop above] (4b);
		\draw (5) edge[loop above] (5);
		\draw (6) edge[loop above] (6);
		\draw (6b) edge[loop above] (6b);
		\draw (7) edge[loop above] (7);
		\draw (8) edge[loop above,min distance=6mm,in=55,out=125,looseness=10] (8);
		\draw (9) edge[loop above] (9);
		\draw (10) edge[loop above] (10);
		\draw (11) edge[loop above] (11);
\end{tikzpicture}\end{figure}

\begin{lem}\label{equivsav}
	Consider the Schreier graphs of $F$ for its action on $\Fr$ (see Figure~\ref{sav}). We denote left (respectively right) branch the subgraph of the vertices $v$ for which any geodesic between $v$ and $\frac{5}{8}$ passes trough $\frac{13}{16}$ (respectively $\frac{9}{16}$). On the figure, those are the left and right branches of the tree, along with the rays starting at them. Then each branch can be embedded as a labelled graph into the other.
\end{lem}

Remark that stronger results of self-similarity of this graph have already been observed, see for example~\cite[Section~5.F]{kaimanovichthompson}.

\begin{proof}
	Each branch is a labelled tree, and thus an equivariant embedding is uniquely defined by the image of the root. We choose the image of $\frac{13}{16}$ to be $\frac{25}{32}$. This defines an embedding of the left branch into the right one. Similarly, choosing $\frac{11}{16}$ as the image of $\frac{9}{16}$ defines an embedding of the right branch into the left one.
\end{proof}

This implies:

\begin{lem}\label{nontriv}
	Fix a measure on $F$, the support of which generates $F$ as a semi-group such that the induced random walk on he dyadic numbers almost surely converges towards an end of the graph. Then the exit measure on the end space is not trivial.
\end{lem}

\begin{proof}
	We decompose the end space into five sets: two sets containing respectively the ends of the left or the right branch, and three sets that are the singletons corresponding to the rays at $\frac{5}{8}$ and $\frac{3}{4}$. The rays have equivariant embeddings into the branches. Combining with Lemma~\ref{equivsav}, this means that any of those five sets can be equivariantely embedded into another one. In particular, if the restriction of the exist measure on one of them has non-zero mass, then by transitivity the restriction on the embedding also has non-zero mass.
\end{proof}

\bibliographystyle{plain}
\bibliography{auto,manual}

\begin{thebibliography}{10}

\bibitem{var}
Paolo Baldi, No{\"{e}}l Lohou{\'{e}}, and Jacques Peyri{\`{e}}re.
\newblock {Sur la classification des groupes r{\'{e}}currents}.
\newblock {\em C. R. Acad. Sci. Paris S{\'{e}}r. A-B}, 285(16):A1103----A1104,
  1977.

\bibitem{bartholdi}
Laurent Bartholdi.
\newblock {Amenability of Groups and G-Sets}.
\newblock In Val{\'{e}}rie Berth{\'{e}} and Michel Rigo, editors, {\em
  Sequences, Groups, and Number Theory}, pages 433--544. Springer International
  Publishing, Cham, 2018.

\bibitem{thomsoncfp}
James~W. Cannon, William~J. Floyd, and Walter~R. Parry.
\newblock {Introductory notes on Richard Thompson's groups}.
\newblock {\em Enseignement Math{\'{e}}matique}, 42:215--256, 1996.

\bibitem{erschsubexp}
Anna Erschler.
\newblock {Boundary behavior for groups of subexponential growth}.
\newblock {\em Annals of Mathematics}, 160(3):1183--1210, nov 2004.

\bibitem{Grigorchuk2012}
Rostislav Grigorchuk, Vadim~A. Kaimanovich, and Tatiana Nagnibeda.
\newblock {Ergodic properties of boundary actions and the Nielsen-Schreier
  theory}.
\newblock {\em Advances in Mathematics}, 230(3):1340--1380, 2012.

\bibitem{MR0450121}
Jonathan~L. Gross.
\newblock {Every connected regular graph of even degree is a {S}chreier coset
  graph}.
\newblock {\em J. Combinatorial Theory Ser. B}, 22(3):227--232, 1977.

\bibitem{juscof}
Kate Juschenko.
\newblock {A remark on Liouville property of strongly transitive actions}.
\newblock preprint, \arXiv{1806.02753}, jun 2018.

\bibitem{juszheng}
Kate Juschenko and Tianyi Zheng.
\newblock {Infinitely supported Liouville measures of Schreier graphs}.
\newblock {\em Groups, Geometry, and Dynamics}, 12:911--918, 2018.

\bibitem{Kaimanovich1991}
Vadim~A. Kaimanovich.
\newblock {Poisson boundaries of random walks on discrete solvable groups}.
\newblock In Herbert Heyer, editor, {\em Probability Measures on Groups X},
  pages 205--238. Springer US, Boston, MA, 1991.

\bibitem{kaimanovichthompson}
Vadim~A. Kaimanovich.
\newblock {Thompson's group {$F$} is not {L}iouville}.
\newblock In Tullio Ceccherini-Silberstein, Maura Salvatori, and Ecaterina
  Sava-Huss, editors, {\em Groups, Graphs and Random Walks}, London
  Mathematical Society Lecture Note Series, pages 300--342. Cambridge
  University Press, Cambridge, 2017.

\bibitem{kaimpoisson}
Vadim~A. Kaimanovich and Anatoly~M. Vershik.
\newblock {Random walks on discrete groups: boundary and entropy}.
\newblock {\em The Annals of Probability}, 11(3):457--490, 1983.

\bibitem{Kesten1959}
Harry Kesten.
\newblock {Symmetric random walks on groups}.
\newblock {\em Transactions of the American Mathematical Society},
  92(2):336--336, 1959.

\bibitem{thesis}
Paul-Henry Leemann.
\newblock {\em {On subgroups and Schreier graphs of finitely generated
  groups}}.
\newblock PhD thesis, Univ. Genève, aug 2016.

\bibitem{meier}
John Meier.
\newblock {\em {Groups, Graphs and Trees}}.
\newblock Cambridge University Press, Cambridge, 2008.

\bibitem{mischenko2015}
Pavlo Mishchenko.
\newblock {Boundary of the action of Thompson group F on dyadic numbers}.
\newblock preprint, \arXiv{1512.03083}, 2015.

\bibitem{qtheory}
John Rhodes and Benjamin Steinberg.
\newblock {\em {The q-theory of Finite Semigroups}}.
\newblock Springer Monographs in Mathematics. Springer US, Boston, MA, 2009.

\bibitem{rosenblatt}
Joseph Rosenblatt.
\newblock {Ergodic and mixing random walks on locally compact groups}.
\newblock {\em Mathematische Annalen}, 257(1):31--42, 1981.

\bibitem{slav10}
Dmytro Savchuk.
\newblock {Some graphs related to Thompson's group $F$}.
\newblock In Oleg Bogopolski, Inna Bumagin, Olga Kharlampovich, and Enric
  Ventura, editors, {\em Combinatorial and geometric group theory}, pages
  279--296, Basel, 2010. Birkh{\"{a}}user Basel.

\bibitem{Schneider2020}
Friedrich~Martin Schneider and Andreas Thom.
\newblock {The Liouville property and random walks on topological groups}.
\newblock {\em Commentarii Mathematici Helvetici}, 95(3):483--513, 2020.

\bibitem{hz}
Bogdan Stankov.
\newblock {Non-triviality of the Poisson boundary of random walks on the group
  $H(\mathbb{Z})$ of Monod}.
\newblock {\em Ergodic Theory and Dynamical Systems}, pages 1--30, 2019.

\bibitem{woess2000random}
Wolfgang Woess.
\newblock {\em {Random Walks on Infinite Graphs and Groups}}.
\newblock Cambridge Tracts in Mathematics. Cambridge University Press, 2000.

\end{thebibliography}

\end{document}